\documentclass[journal]{IEEEtran}
\usepackage{lscape}
\usepackage{bigints}
\usepackage{graphicx} 
\usepackage{cite}
\usepackage{psfrag}
\usepackage{amsfonts}
\usepackage{amssymb}
\usepackage{amsmath,mathabx}
\usepackage{graphicx}
\usepackage{psfrag}
\usepackage{ae,epsfig}
\usepackage{verbatim} 
\usepackage{enumerate} 
\usepackage{epstopdf}
\usepackage{tikz}
\usetikzlibrary{shapes,arrows}
\usetikzlibrary{positioning,fit}
\usepackage{color}
\usetikzlibrary{positioning}
\usepackage{subcaption}
\usepackage{tikz}
\usetikzlibrary{shapes,arrows}
\usetikzlibrary{positioning,fit}
\usepackage{bm}
\usepackage{bbm}
\usetikzlibrary{calc}
\usetikzlibrary{matrix}
\usepackage[utf8]{inputenc}

\usepackage[UKenglish]{babel}
\usepackage{fullpage}
\usepackage{setspace}

\newcommand{\tf}{t_{\rm f}}

\usepackage{epsfig,amsmath,amsfonts,amsthm,mathrsfs,enumitem,amssymb,tasks,multicol,tcolorbox}
\usepackage{thmtools,subcaption}

\newtheorem{theorem}{Theorem}[section]

\newtheorem{corollary}[theorem]{Corollary}

\theoremstyle{definition}
\newtheorem{definition}[theorem]{Definition}
\theoremstyle{remark}
\newtheorem{remark}[theorem]{Remark}

\def\XXint#1#2#3{{\setbox0=\hbox{$#1{#2#3}{\int}$}
     \vcenter{\hbox{$#2#3$}}\kern-.5\wd0}}

\newcommand{\pdiff}[2]{\frac{\partial #1}{\partial #2}}

\newcommand{\eps}{\varepsilon}

\newcommand{\mR}{\mathbb R}

\newcommand{\rd}{{\rm d}}

\usepackage{tikz}
\usetikzlibrary{automata,intersections,shapes,arrows,calc,positioning,decorations}
\usepackage{pgfplots}
\pgfplotsset{compat=newest}
\usepackage{tkz-fct}
\usetikzlibrary{matrix}
\usepgfplotslibrary{fillbetween}

\tikzstyle{sum} = [draw, fill=white, thick, circle, inner sep=2pt,minimum size=2pt]
\tikzstyle{wheel} = [draw, fill=black, thick, circle, inner sep=3pt,minimum size=3pt]

\tikzset{
    dimen/.style={
        <->,
        >=latex,
        thick,
        every rectangle node/.style={
            fill=white,
            midway,
            font=\sffamily,
            }
        },%
}
\numberwithin{equation}{section}

\usepackage{pifont}

\usepackage[urlcolor = black, colorlinks = true, citecolor=black, linkcolor = black]{hyperref}
 \usepackage{flushend} 

\begin{document}
\author{Ross Drummond, Nicola E. Courtier, David A. Howey, Luis D. Couto, Chris Guiver
\thanks{ R. Drummond is with the Department of Automatic Control and Systems Engineering, University of Sheffield,  Mappin St, Sheffield, S1 3JD, United Kingdom. Email: {\tt ross.drummond@sheffield.ac.uk.}

N. E. Courtier and D. A. Howey are with the Department of Engineering Science, University of Oxford,  17 Parks Road, OX1 3PJ, Oxford, United Kingdom, Email: {\tt \{david.howey, nicola.courtier\}@eng.ox.ac.uk}.

L D. Couto is with the Department of Control Engineering and System Analysis, Universit\'{e} libre de Bruxelles, Brussels, B-1050, Belgium. Email: {\tt luis.daniel.couto.mendonca@ulb.be}.

C. Guiver is with the School of Engineering \& the Built Environment, Edinburgh Napier University, Edinburgh, EH10 5DT, UK. Email: {\tt c.guiver@napier.ac.uk}.

\date{\today}

Ross Drummond would like to thank the Royal Academy of Engineering for funding through a UK Intelligence Community research fellowship.

Chris Guiver would like to thank the Royal Society of Edinburgh (RSE) for funding through a RSE Personal Research Fellowship.

Nicola Courtier and David Howey were supported by the EPSRC Faraday Institution Multiscale Modelling project (EP/S003053/1, grant number FIRG025).
}
}

\IEEEoverridecommandlockouts


\title{Constrained optimal control of monotone  systems with applications to battery fast-charging}

\maketitle
\thispagestyle{plain}
\pagestyle{plain}

\begin{abstract}
Enabling fast charging for lithium ion batteries is critical to accelerating the green energy transition. As such, there has been significant interest in tailored fast-charging protocols computed from the solutions of constrained optimal control problems. Here, we derive necessity conditions for a fast charging protocol  based upon monotone control systems theory.

\end{abstract}


\section*{Erratum}
\textbf{The first version of this manuscript was submitted for publication and uploaded to Arxiv in January 2023. In the autumn of 2023 we found an error in the original Corollary I.3., which was not true as stated. The error was a genuine mistake, but unfortunately has consequences throughout the rest of the work. We have withdrawn the manuscript from peer review and are working on a corrected version. We apologise to any readers who have used the original, flawed version of this work.}

\textbf{Corollary I.3 has been corrected in the current work, and the manuscript updated accordingly.}

\section*{Introduction}
The lithium ion battery is one of the leading energy storage devices of the green energy transition.  In fact, for several emerging applications, such as electric vehicles and grid storage, it is often the performance of the battery which is the limiting factor. For this reason, there has been a growing demand for ``better batteries'', as in those with increased energy densities and lifespans, that can be manufactured cheaply, at scale, with only a minimal environmental footprint and which, crucially, can be fast charged. Of these, it is perhaps  fast charging  which is the most pressing issue holding back the widespread adoption of  electric vehicles (EVs). Whilst most EVs now allow for charging at speeds of around 20-80\% state of charge in under 30 minutes,  further advances are needed before EV fast charging can be made as quick as refilling an internal combustion engine vehicle  and can be implemented on an everyday basis without impacting the health of the battery.

Many factors have been observed to influence battery fast charging performance~\cite{tomaszewska2019lithium,wassiliadis2021review}, but, of these, one of the most critical is, simply, the profile of the applied charging current~\cite{keil2016charging}. Computing an optimal fast-charging current requires solving a constrained optimal control problem to minimise the charging time whilst ensuring the cell remains both safe and healthy. Significant gains can be delivered in this way. For example Couto {\em et al.}~\cite{couto2021faster} showed that, compared to 2C constant current-constant voltage (CC-CV) charging, an optimised charging current could reduce the charging time by 22\% whilst reducing the capacity loss by 26\%. 

{ A wide variety of schemes} for computing fast-charging currents  have been developed, including variations of  CC-CV charging~\cite{maia2019expanding} and constant-temperature effects \cite{patnaik2018closed},  pack-level considerations~\cite{ouyang2019optimal}, and the use of sinusoidal currents~\cite{vincent2017system}. By way of related literature, the papers~\cite{perez2017optimal,suthar2014optimal,romagnoli2017computationally,ECC2022} each address the fast charging problem from a model-based perspective, with constraints to maintain the cell's health using different combinations of models and optmization methods. The papers~\cite{zou2017electrochemical,zou2018model,tian2020real} are also model-based and study various model predictive control (MPC) schemes for a range of cell models. The papers\cite{chen2022data}, \cite{jiang2022fast} and~\cite{sieg2019fast} present data-driven fast charging schemes, namely via a data-driven Bayesian Optimisation approach, an experimentally-derived CC-CV type-protocol designed to limit Li-ion plating, and so-called Data-EnablEd Predictive Control (DeePC), respectively. The importance of uncertainties is discussed in~\cite{cai2022fast}. In addition to these references, we refer the reader to the reviews of fast charging~\cite{tomaszewska2019lithium,keil2016charging,wassiliadis2021review,dufek2022developing}. 

As well as electrochemical considerations, the cell temperature  has also been found to play a significant role in fast-charging performance~\cite{mohtat2021algorithmic,f2021fast}, especially as a means to limit Li-plating \cite{yang2019asymmetric}. Two of the main issues with such  model-driven studies are  their computational complexity, and  the inherent biases and limitations of models. To overcome these issues, there has been recent interest in applying data-driven methods~\cite{attia2020closed,jiang2022fast, park2022deep,park2020reinforcement}, however, it can be argued that this approach is { currently} limited by the availability of  real-world data. 

In the studies mentioned above, various solutions are presented for the fast-charging problem. However, these often lacked detailed mathematical analysis, making it difficult to discern why the computed optimal controls took the forms that they did. Such insight could be gained by applying  methods from control theory. One notable study on characterising the solutions of battery fast-charging problems is~Park {\em et al.}~\cite{park2020optimal} which showed, by solving the maximum principle equations by hand, that a form of CC-CV charging is optimal for a constrained problem involving the  Li-ion battery single particle model (SPM). In the context of this paper, the results of
~\cite{park2020optimal}  are important because they contain, to the best of our knowledge, the first methodical reasoning to explain why experimentally-derived protocols, such as CC-CV, have proven so successful in practice.

Building upon the ideas of~
\cite{park2020optimal},  the present paper generalises that analysis by presenting conditions for which the solution of optimal fast-charging problems can be explicitly defined. {  In overview}, here it is shown that provided the underlying battery model admits certain \textit{monotonicity} assumptions, then the solutions of the constrained optimal control problems necessarily ride one of the constraints at some time instant during the charge. We use this result to highlight the applicability of ``bang-and-ride control''. Informally, bang-and-ride control is a policy that is either maximal (bang), or as large as possible whilst not violating any state constraints (ride).

Underlying this result is the theory of {\em monotone control systems}. The analysis of these systems dates back to the seminal work of Angeli and Sontag~\cite{angeli2003monotone} where their potential for analysing biological systems was highlighted, in particular for detecting multi-stability. These systems generalise the powerful and well-studied concept of {\em monotone dynamical systems}, see~Hirsch and Smith~\cite{MR1319817} and Smith~\cite{MR2182759}, to a natural control setting where input and output variables are present. For a recent survey on monotone control systems we refer the reader to~Smith~\cite{smith2017monotone}.
Loosely speaking, a monotone control system is  one where any ordering of initial states and inputs is preserved by the states over time. In particular, the value of the states increases if the input increases, a structural feature which enables several powerful systems analysis techniques to be applied. To be clear, here, the term ``increases'' means with respect to a partial order determined by a so-called positive cone---such as the nonnegative orthant in real Euclidean space.

A related concept is that of so-called {\em positive dynamical systems}~\cite{MR2655815,MR1019319}, or {\em positive control systems}~\cite{MR1784150} again when inputs and outputs are present. With positive systems, their defining feature is that of invariance---meaning solutions which start in a positive cone remain in that cone over all time, reflecting the fact that state variables often correspond to necessarily nonnegative quantities, such as ion concentrations for battery applications. Positive systems are simple examples of monotone control systems that have garnered much interest in the control literature because they often admit linear Lyapunov functions, and scalable model-order reduction~\cite{kawano2019data} and feedback controller design~\cite{rantzer2015scalable} methods for them exist.

To the best of our knowledge, we are unaware of any existing papers that have directly exploited monotonicity for battery models. This omission is surprising owing to the inherent monotonicity of many of these models. Indeed, in general,  it has been observed that increasing the applied current causes the various voltages and temperatures that form the battery models' state variables to also increase, a monotonic relation. It is shown in this paper how this structural feature of battery models' monotonicity can be exploited to  characterise the solutions of fast charging optimal control problems. 
 
\section{An optimal control problem for monotone control systems}\label{sec:mon}

In this section, the optimal control problems for the monotone control systems under consideration are introduced and shown to be optimised by bang-and-ride control.

\subsection{Problem set-up} 
Consider the system of controlled nonlinear differential equations
\begin{subequations}\label{sys:oc_problem}
\begin{equation}\label{eq:model}
 \dot x(t) = f(x(t),u(t)), \quad x(0) = x_0\,,
\end{equation}
for the given function~$f : \mR^n \times \mR^m \to \mR^n$ which is assumed continuously differentiable. Here, as usual,~$x(t) \in \mR^n$ is the state variable, with initial state $x_0 \in \mR^n$, and~$u(t)\in \mR^m$ is the input variable. 

We shall assume that, for all~$\tau>0$, all~$x_0 \in \mR^n$ and all measurable, locally essentially bounded~$u$, a unique {\em solution}~$x$ of~\eqref{eq:model} exists, that is, an absolutely continuous function~$x : [0,\tau] \to \mR^n$ which satisfies~\eqref{eq:model} almost everywhere. Given such~$x_0$ and~$u$, we let~$x(t) = x(t;x_0,u)$ denote the solution of~\eqref{eq:model} at time~$t \in [0,\tau]$, and we call the pair~$(u,x)$ a {\em trajectory} of~\eqref{eq:model}. In light of the standing assumption that~$f$ is continuously differentiable, existence of unique solutions of~\eqref{eq:model} is guaranteed from known results under an additional mild boundedness assumption on $f$ \cite[Theorem 54,  Proposition C.3.8]{sontag2013mathematical}.

Given a running-cost function~$L: \mR^n \times \mR^m \to \mR$, constraint function~$h : \mR^n \times \mR^m \to \mR^s$ {  with components $h_k$}, and fixed final-time~$\tf >0$, consider the cost functional~$J$ given by
\begin{equation}\label{eq:cost}
J(x_0, u) := \int_0^{\tf} L(x(t),u(t))\, \rd t\,,
\end{equation}  
and the mixed constraint
\begin{equation}\label{eq:constraint}
h(x(t),u(t)) \geq 0\,
\end{equation} 
\end{subequations}
(understood component-wise). {  Equation~\eqref{eq:constraint} is a mathematically convenient method of capturing constraints on the state, input, or mixed constraints}. The positive integer~$s$ denotes the number of mixed constraints. A trajectory $(u,x)$ of~\eqref{eq:model} with $x$ defined on $[0,\tf]$ and which satisfies the constraint~\eqref{eq:constraint} on $[0,\tf]$ is called an {\em admissible trajectory}. We shall say that constraint $k$ is {\em engaged} if $h_k(x(t),u(t)) = 0$.

The corresponding optimal control problem is to maximise~$J$ given by~\eqref{eq:cost}, over all piecewise continuous~$u : [0,\tf] \to \mR^m$, subject to~\eqref{eq:model} and the constraint~\eqref{eq:constraint}.

\subsection{Monotonicity properties}

Monotonicity of~\eqref{sys:oc_problem} plays a key role in the present work. Here, we gather the required notation, terminology and hypotheses. 
For $x, y \in \mR^n$, with components $x_i$, $y_i$, we write
\[ \begin{aligned}
x &\leq y \quad \text{if $x_{i} \leq y_{i}$  for all $i = 1,2,\dots,n$}, \\
x &< y \quad \text{if $x \leq y$ and $x \neq y$}, \\
x &\ll y \quad \text{if $x_{i} < y_{i}$  for all $i = 1,2,\dots,n$}\,. 
\end{aligned}\]
We shall also use the symbols $\geq$, $>$ and $\gg$, defined analogously. Note that the symbol $\lll$ is also used in the literature instead of $\ll$.

As usual, we let $L^\infty_{\rm loc}(\mR_+,\mR^m)$ denote all measurable and locally essentially bounded functions $\mR_+ \to \mR^m$. For $u_1, u_2 \in L^\infty_{\rm loc}(\mR_+,\mR^m)$, we write $u_1 \leq u_2$ and $u_1 < u_2$ if, respectively,
\begin{align*}
    u_1(t) &\leq u_2(t) \quad \text{for almost all $t \geq 0$,} \\
\text{and} \quad  u_1(t) &< u_2(t) \quad \text{for almost all $t \geq 0$.}
\end{align*} 
%
%
We recall two definitions from Angeli and Sontag~\cite{angeli2003monotone} and~\cite{angeli2004multi} for monotone control systems and their excitability. 
\begin{definition}
The controlled differential equation~\eqref{eq:model} is called a {\em monotone control system} if, for all $t\geq0$, all $\xi_1, \xi_2 \in \mR^n$ and all $u_1,u_2 \in L^\infty_{\rm loc}(\mR_+,\mR^m)$, the implication 
 \begin{align*} & \xi_1 \leq \xi_2 ~ \text{and} ~ u_1 \leq u_2 
~  \Rightarrow ~ x(t; \, \xi_1, u_1) \leq x(t; \, \xi_2, u_2)\,,\end{align*}
holds ($\xi_j$ is a dummy variable for the initial state).

{  Positive linear systems,} that is,
\begin{equation}\label{eqn:lin_ss}
\dot x(t) = Ax(t) + B u(t)\,,
\end{equation}
for Metzler matrix $A$ (every off-diagonal entry nonnegative) and componentwise nonnegative matrix $B$ comprise a simple class of monotone control systems; see,~\cite[Section VIII]{angeli2003monotone}. Positive linear systems are well-studied objects and we refer the reader to, for example~\cite{MR2655815,MR1784150} for further background.

The controlled differential equation~\eqref{eq:model} is called {\em excitable} if, for  all $\xi \in \mR^n$, and all $u_1,u_2 \in L^\infty_{\rm loc}(\mR_+,\mR^m)$ with $u_1 < u_2$, it follows that 
\[ x(t; \, \xi, u_1) \ll x(t; \, \xi, u_2) \quad \forall \: t >0\,. \]
\end{definition}
The concept of excitability intuitively means that every input has an effect on every state variable, perhaps indirectly. A graphical test for excitability is given in Angeli and Sontag~\cite[Theorem 4]{angeli2004multi}.

%
%
It is convenient to formulate a number of hypotheses on the model~\eqref{eq:model}, cost functional~\eqref{eq:cost} and constraint~\eqref{eq:constraint}. 
\begin{enumerate}[label = {\bfseries (H\arabic*)}, ref = {\rm \bfseries (H\arabic*)}, start = 1,itemsep = 0.5ex]
%
%
\item \label{ls:H1_f} \eqref{eq:model} is a monotone control system.
%
%
\item \label{ls:H2_L} The running-cost function~$L$ is continuous and monotone in the sense that 
{\small\begin{align*}
   L(\xi_1, v_1) \leq L(\xi_2, v_2) \quad &\forall \: (\xi_1,v_1), (\xi_2, v_2)  \in \mR^n \times \mR^m \quad 
    \\
     & \quad \text{with} \quad (\xi_1,v_1) \leq (\xi_2, v_2)\,,
\end{align*} with $\xi_j$ being a dummy variable for the state and $v_j$ for the input. }
%
%
\item \label{ls:H3_L_strict} One of the following holds:
\begin{itemize}
    \item for all $\xi \in \mR^n$ 
    {\small \[ L(\xi, v_1) < L(\xi, v_2) \quad \forall \: v_1, v_2 \in \mR^m, \; v_1 < v_2\]}
    \item \eqref{eq:model} is excitable and for all $v \in \mR^m$
{\small
\[ L(\xi_1, v) < L(\xi_2, v) \quad \forall \: \xi_1, \xi_2 \in \mR^n, \; \xi_1 \ll \xi_2\,.\]
}
\end{itemize}
%
%
\item \label{ls:H4_h} The constraint function~$h$ is continuous, and every  component $h_k$ is non-increasing in each input component $u_j$. 
\end{enumerate}

The functions $f$, $L$ and $h$ do not need to be defined on all of $\mR^n \times \mR^m$. They could instead be defined on some subset $X \times U$ of $\mR^n \times \mR^m$ (with minor technical assumptions, as described in~Angeli and Sontag~\cite{angeli2003monotone}), and have their continuity and monotonicity properties there. The current choice of $X =\mR^n$ and $U = \mR^m$ is used  to reduce the volume of notation and assumptions introduced.

%
%
Since $f$ in~\eqref{eq:model} is assumed continuously differentiable, the result of  Angeli and  Sontag~\cite[Proposition III.2]{angeli2003monotone} yields that a necessary and sufficient condition for~\eqref{eq:model} to be a monotone control system is
\[ \pdiff{f_j}{\xi_i}(\xi,v) \geq 0 \quad \forall \: \xi \in \mR^n, \; \forall \: v \in \mR^m, \; \forall \: i \neq j\,,\]
and
\[ \pdiff{f_j}{v_i}(\xi,v) \geq 0 \quad \forall \: \xi \in \mR^n, \; \forall \: v \in \mR^m, \; \forall \: i, j\,.\]
%
%
%
Hypotheses~\ref{ls:H1_f} and~\ref{ls:H2_L} capture monotonicity of the control system~\eqref{eq:model} and cost-functional~$J$ in~\eqref{eq:cost}, respectively. 
It is intuitively clear that, under these two assumptions, the cost functional is maximised by making~$u$, and hence~$x$, as large as possible (componentwise). 
However, assumption~\ref{ls:H4_h} acts to bound~$u$, as it captures that~$h$ is non-increasing in~$u$, so that in practice arbitrarily large inputs are not admissible.

%
%
However, note that assumptions~\ref{ls:H1_f} and~\ref{ls:H2_L} still allow for somewhat degenerate situations, such as $f = 0$ and $L=0$, where the cost is always equal to zero, so zero is the optimal cost. Since every piecewise continuous input is admissible, and also optimal, it follows that optimisers are not unique. 
Therefore, hypothesis~\ref{ls:H3_L_strict} seeks to address this degeneracy by enforcing ``strict'' monotonicity, either entirely in the running-cost function $L$, or jointly across the running-cost and dynamics for $x$ via the excitability condition. The two displayed inequalities in~\ref{ls:H3_L_strict} collapse to the usual notion of strictly increasing for functions of a scalar variable, for fixed other variable. However, observe that in the general multivariable case ($n,m >1$) there is a deliberate asymmetry between these inequalities. {  Roughly speaking}, for the applications we shall consider, not every component of $\xi \in \mR^n$, representing the state variable, need appear in $L$, which may lead to $L(\xi_1,v) = L(\xi_2,v)$ for certain $\xi_1, \xi_2 \in \mR^n$ even though $\xi_1< \xi_2$. Consequently, the requirement that $L(\xi_1,v) < L(\xi_2,v)$ only when $\xi_1< \xi_2$ places too restrictive a constraint on $L$ for our purposes, and has been replaced by the second part of~\ref{ls:H3_L_strict}. 

%
%
The condition~\ref{ls:H4_h} prevents lower bounds for the input variable(s), such as
\begin{equation}\label{eq:h_lower_bound}
    \underline{u} \leq u(t) \quad \forall \: t \in [0, \tf]\,,
\end{equation} 
appearing in the mixed-constraint~\eqref{eq:constraint}, for some given $\underline{u} \in \mR^m$, even though such lower bounds are common in practical applications. Indeed, an assumption of the form~\eqref{eq:h_lower_bound} would be encoded in the mixed constraint $h$ via
\[ h_\mathcal{I}(\xi, v) = v - \underline{u} \quad \forall \: \xi \in \mR^n, \; \forall \: v \in \mR^m\,, \]
(some collection of components $\mathcal{I}$), which is an increasing function of $v$.
The reason for the present formulation of~\ref{ls:H4_h} is, {  roughly speaking}, 
that lower bounds are not engaged 
by optimal controls. The following analysis is made easier by omitting these types of constraints.

The next theorem is the main result of this section.

\begin{theorem}\label{thm:oc_monotone}
Fix $x_0 \in \mR^n$ and $\tf >0$, and consider the cost functional $J$ given by~\eqref{eq:cost} subject to~\eqref{eq:model}. 
Let $u_1, u_2 \in L^\infty_{\rm loc}(\mR_+,\mR^m)$ and assume that~\ref{ls:H1_f} and~\ref{ls:H2_L} hold. {  Then} 
\begin{enumerate}[label = {\rm (\arabic*)}]
    \item \label{ls:monotone_cost_1} The function $v \mapsto J(x_0,v)$ is non-decreasing, in the sense that 
\begin{align*} J(x_0, u_1) \leq J(x_0, u_2) \quad  
\text{ when } \quad u_1\leq u_2.
\end{align*}
    \item \label{ls:monotone_cost_2} If, additionally, hypothesis~\ref{ls:H3_L_strict} holds and $u_1, u_2$ are continuous at some $t_* \in [0, \tf]$ with $u_1(t_*) < u_2(t_*)$, then~$J(x_0, u_1) < J(x_0, u_2)$.
\end{enumerate}
\end{theorem}
Strictly speaking, elements of $L^\infty_{\rm loc}(\mR_+)$ are equivalence classes of functions, and do not necessarily have well-defined point values. Thus, in statement~\ref{ls:monotone_cost_2}, we mean that $u_1$ and $u_2$ have representatives on a neighbourhood of $t_*$ which are continuous at $t_*$. This will always be the case if, for example, $u_1$ and $u_2$ are themselves piecewise continuous functions which are continuous at $t_*$.

{  Statement~\ref{ls:monotone_cost_2} of Theorem~\ref{thm:oc_monotone} is one way of making precise and ensuring the desired implication
\[ ``u_1 < u_2 \quad \Rightarrow \quad J(x_0, u_1) < J(x_0,u_2)\text{''} \]
the essential challenge being that a single pointwise inequality $u_1(t) < u_2(t)$ need not ensure the inequality of integrals $J(x_0, u_1) < J(x_0,u_2)$. Our approach is to use continuity of $u_1$ and $u_2$.}

\begin{proof}[Proof of Theorem~\ref{thm:oc_monotone}]
For the proof of both statements, fix~$x_0 \in \mR^n$, and let~$u_1, u_2 \in L^\infty_{\rm loc}(\mR_+,\mR^m)$ be such that~$u_1 \leq u_2$.

\ref{ls:monotone_cost_1}  From the monotonicity of~\eqref{eq:model}, it follows that
\[ x(t; x_0, u_1) \leq x(t;x_0,u_2) \quad \text{for all $t \in [0, \tf]$}\,.\]
The monotonicity assumption~\ref{ls:H2_L} on $L$ now gives
\begin{equation}\label{eq:L_monotone_1}
    L(x(t; x_0, u_1), u_1(t)) \leq L(x(t; x_0, u_2), u_2(t))
\end{equation} 
for almost all $t \in [0, \tf]$. 
Integrating both sides of~\eqref{eq:L_monotone_1} over $[0,\tf]$ yields the desired inequality.

\ref{ls:monotone_cost_2} The hypotheses of statement~\ref{ls:monotone_cost_1} hold, so that in particular the inequality~\eqref{eq:L_monotone_1} holds. 

By the assumed continuity of $u_1$ and $u_2$ at $t_*$, there exist {  $\delta >0$} and $\eps_0 \in \mR^n$ with $\eps_0 >0$ and such that
\begin{equation}\label{eq:u_inequality}
 u_1(t) \,{  <}\, u_2(t) + \eps_0 \quad \forall \: t \in [t_1 , t_2]\,,
\end{equation}
where $t_1 := t_* - \delta$ and $t_2 : = t_* + \delta$. (We adjust this interval accordingly if $t_* =0$ or $t_* = \tf$.)

Assume that the first item in hypothesis~\ref{ls:H3_L_strict} holds, {  briefly}, that $L$ is strictly increasing in its second variable.  By continuity of $L$ and the inequality~\eqref{eq:u_inequality}, there exists $\eps_1 >0$ such that
{\small
\begin{align}
    \hspace{-1ex}L(x(t; x_0, u_1), u_1(t)) & \leq L(x(t; x_0, u_1), u_2(t)) + \eps_1 \notag \\
    & \leq L(x(t; x_0, u_2), u_2(t)) + \eps_1 \,, \label{eq:L_inequality}
\end{align}}
for all $t \in  [t_1, t_2]$. Here we have used~\ref{ls:H2_L} to obtain the second inequality.  Integrating the above over $t \in [0,\tf]$, combined with~\eqref{eq:L_monotone_1}, yields that $J(x_0, u_1) < J(x_0, u_2)$, as required. 

Now assume, instead, that the second item in hypothesis~\ref{ls:H3_L_strict} holds, namely that~\eqref{eq:model} is excitable and, {  (thereabouts)} $L$ is strictly increasing in its first variable. The inequality~\eqref{eq:u_inequality} may be expressed as $u_1 < u_2$ on $[t_1,t_2]$, and excitability yields that
\begin{align*}
    x(t; x(t_1,\xi,u_1), u_1) \ll x(t; \, &x(t_1,\xi,u_2), u_2) \\ 
    & \forall \: t \in (t_1, t_2]\,.
\end{align*}
Invoking the current hypothesis on $L$, the inequality~\eqref{eq:L_inequality} again holds, now on $(t_1, t_2]$. Integrating over $[0,\tf]$ as before gives $J(x_0, u_1) < J(x_0, u_2)$, completing the proof.
\end{proof}

%
%
A consequence of the above theorem is the following necessary condition for an optimal control.

\begin{corollary}\label{cor:necessary}
Consider the optimal control problem~\eqref{sys:oc_problem} and assume that hypotheses~\ref{ls:H1_f}--\ref{ls:H4_h} hold. Fix $x_0 \in \mR^n$ and let $(u,x)$ denote an admissible trajectory  with piecewise continuous $u$. If there exists $t_0 < \tf$ such that
\begin{equation}\label{eq:h_strongly_positive}
    h(x(t),u(t)) \gg 0 \quad \forall \: t \in [t_0, \tf]\,,
\end{equation} 
then $(u,x)$ is not optimal.
\end{corollary}

\begin{proof}[Proof of Corollary~\ref{cor:necessary}]
Let~$(u,x)$ denote an admissible trajectory of~\eqref{eq:model} with piecewise continuous $u$, and suppose that~\eqref{eq:h_strongly_positive} holds. By continuity of $h$, and the continuous dependence of $x$ on $u$, we may choose $t_1 \in (t_0,\tf)$ which is a point of continuity of $u$, and choose (sufficiently small) piecewise continuous $u_*$ such that: 
  $(u+u_*,x_*)$ is an admissible trajectory of~\eqref{eq:model}, for some corresponding state $x_*$ with $x_*(0) = x_0$, and;
   $u_*$ is continuous at $t$ with $u_*(t) >0$.
An application of statement~\ref{ls:monotone_cost_2} of Theorem~\ref{thm:oc_monotone} yields that $J(x_0, u) < J(x_0, u + u_*)$, and hence~$(u,x)$ is not optimal.
\end{proof}

\begin{remark}
In the original arxiv submission of this article, the incorrect version of the above result was erroneously used to indicate the optimality of bang-and-ride control. We apologise for any inconvenience this may have caused. 
\end{remark}

Some additional remarks are in order. The contrapositive of the above statement is that, if an admissible trajectory $(u,x)$ of~\eqref{eq:model} with piecewise continuous~$u$ is optimal, then for every $t_0 \in (0,\tf)$, there is some $t \in [t_0, \tf]$ such that
\begin{equation} \label{eq:optimal}
\left.\begin{aligned}
&h(x(t),u(t)) \geq 0 \\
~\text{and} ~ &h(x(t),u(t)) \not \gg 0 \end{aligned} \right\}\,.
\end{equation}
Condition~\eqref{eq:optimal} does not, by itself, imply that an optimal trajectory satisfies~\eqref{eq:optimal} at {\em every} $t \in [0, \tf]$. It is possible that an optimal trajectory satisfies $h(x(t),u(t))\gg0$ on some proper sub-interval $(t_1,t_2)$ of $[0,\tf]$. As such, Corollary~\ref{cor:necessary} does not yet provide a constructive method of identifying optimal controls. Indeed, what appears to happen in numerical simulations is that a numerically-computed optimal trajectory does not engage any constraint on short intervals between switches of which constraint is engaged.

\begin{remark}
It is noted that a similar ``switching''-type fast charging protocol to the above {  was} recently proposed by Berliner {\em et al.} in  \cite{berlinerfast} and \cite{berliner2022novel}, where the fast charging protocol switched between different ``operating modes'' related to actions such as constant power/temperature charging or enforcing a zero overpotential drop.  Such switching  policies can be understood as instances of ``bang-and-ride'' control (the different operating modes correspond to different constraints being ridden), so Corollary \ref{cor:necessary} can be used to infer why these protocols performed well in~\cite{berlinerfast}. 
\end{remark}

\section{Monotonicity of battery fast-charging problems}\label{sec:mon_batt}

Here, we show that several battery fast-charging problems satisfy the monotonicity assumptions~\ref{ls:H1_f}--\ref{ls:H4_h}.  

\begin{remark}
In the following, the notation that a positive current charges a cell is adopted. This notation contrasts with some of the existing literature, such as~Couto {\em et al.}~\cite{couto2021faster}, where a negative current charges the cell.  This choice of notation does not impact the results and is adopted here to make clear the connection between battery models and monotone control systems. 
\end{remark}

\subsection{Equivalent circuit model}\label{sec:ecm}

Perhaps the most popular form of battery model is the electrical equivalent circuit model of the form shown in  Figure~\ref{fig:ECM} with dynamics given by~\eqref{eqn:lin_ss} where $A := -\text{diag}(0,1/(R_1C_1),\, \dots,\,1/(R_{n-1}C_{n-1}))$, $B := [1/Q,\,1/C_1,\, \dots,\, 1/C_{n-1}]^{\top}$, $C_k$ and $R_k$ are, respectively, the capacitances and resistances of the RC-pairs of Figure~\ref{fig:ECM} and $Q$ is the cell capacity. {  It is remarked that other forms of equivalent circuit models, such as the commonly considered RC ladder networks, are monotone and so can also be studied using the results of this paper. } 

 \begin{figure}[h!]
 \centering
\includegraphics[width=0.40\textwidth]{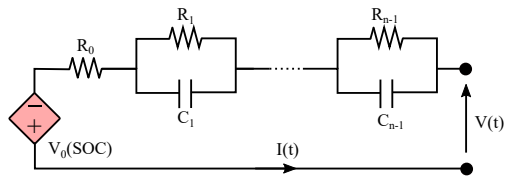}
\caption{{  Circuit diagram of} an equivalent circuit model of a lithium ion battery.}
\label{fig:ECM}
\end{figure}

The cell voltage is defined as
\begin{align}
    v(t) := U(x_1(t)) + \sum_{j= 2}^{n}x_{j} + R_0u(t)\,,
\end{align}
where $R_0$ is the series resistance and $U: \mathbb{R}\to \mathbb{R}$ is the open circuit voltage which is a nonlinear, but generally monotone, function of the state-of-charge, in this case being $x_1(t)$ which is the integral of the applied current $u(t)$ normalised by the cell capacitance.

Bounds on the applied current, the state-of-charge, the voltage drops across each RC-pair and the cell voltage are assumed, as in, for $k = 1,2,\dots, n$
\begin{align}
    u(t) \leq \bar{u}, \quad v(t) \leq \bar{v}, \quad x_k(t) \leq \bar{x}\,. 
\end{align}
{   As an aside, we note that whilst such voltage and temperature limits are examples of common charging constraints, their inclusion is not necessary for the results of the paper. Instead, the focus is on the general class of constraints which satisfy~\ref{ls:H4_h}. }For the running cost, the goal is to maximise the state-of-charge at the end time $t_f$, and so
\begin{align}\label{eqn: L_1}
    L(x(t),u(t)) = \dot{x}_1(t) = \frac{u(t)}{Q}\,.
\end{align}
{  Besides this cost, many other common running costs for fast charging problems are monotone, such as the penalisation of the charge itself and the voltage.} Each of the terms in~\eqref{eqn:lin_ss}-\eqref{eqn: L_1} satisfies the monotonicity assumptions~\ref{ls:H1_f}--\ref{ls:H4_h}.

\subsection{Reduced-order single particle model}

In~\cite{park2020optimal}, it was shown that, by solving the maximum principle equations by hand, a form of CC-CV charging is  optimal for a fast-charging problem involving a reduced-order version of the single-particle model (SPM---one of the standard simplified electrochemical models of a Li-ion battery). For each electrode, this model was obtained from a 3$^{rd}$-order Pad\'{e} approximation of the SPM's transfer function, giving a state-space realisation {  of the form~\eqref{eqn:lin_ss} with $n=3$,
\begin{equation}\label{SPM_park}
A : =\begin{bmatrix} a_1 & 0 & 0 \\ 0 & a_2 & 0 \\ 0 & 0 & 0  \end{bmatrix}, \quad B : = \begin{bmatrix} -b_1 \\ b_2 \\ b_3  \end{bmatrix}
\end{equation}
with} the parameters $a_k \leq 0,$ $b_k~\geq ~0$  depending upon the specific chemistry (values for LCO, NCA and  based cells are stated in~Park {\em et al.}~\cite[Table 1]{park2020optimal}). Furthermore, the fast-charging program considered in~Park {\em et al.}~\cite{park2020optimal} specifies
\begin{align}\label{eqn:other_park}
    L(x(t),u(t)) = x_3(t), ~ u(t) \leq \bar{u}, ~  y(t) \leq \bar{y}\,.
\end{align}
In the above, $y(t)$ is the lithium ion concentration on the surface of the anode's active particle and was defined in~Park {\em et al.}~\cite{park2020optimal} as
\begin{align}\label{eqn:c_s}
    y(t) =
    \begin{bmatrix} -c_1 & c_2 & c_3  \end{bmatrix}
    \begin{bmatrix} {x}_1(t) \\ {x}_2(t) \\ {x}_3(t)  \end{bmatrix}\,,
\end{align}
with the non-negative parameters $c_k \geq 0$ also depending upon the cell chemistry~\cite[Table 1]{park2020optimal}. 

As is, this model does not satisfy the monotone assumption~\ref{ls:H1_f}, owing to the $-b_1$ and $-c_1$ terms in~\eqref{SPM_park} and~\eqref{eqn:c_s}, but it does so after applying the state-space transformation 
\begin{align}\label{eqn:x_trans_park}
    x \mapsto \begin{bmatrix} -1 & 0 & 0 \\ 0 & 1 & 0 \\ 0 & 0 & 1 \end{bmatrix}x\,.
\end{align}
With the above transformation, the system is excitable and hypothesis~\ref{ls:H1_f} is satisfied.

\subsection{Discretised SPM}
The optimality of bang-and-ride control for the SPM  explored in~Park {\em et al.}~\cite{park2020optimal} can be generalised to the case when the model dynamics are spatial discretisations of the underlying diffusion problem instead of being reduced-order approximations of them. Using the superscript $\pm$ to differentiate between the anode and cathode, the SPM's partial differential equations relate to a spatially transformed function of the lithium ion concentrations in the active particles $c^\pm(r,t)$ (varying in both space $r$ and time $t$) diffusing along a particle radius $r \in [0, R^\pm_s(t)]$. {  After applying a spatial transformation on the SPM's spherical diffusion equation, this model's dynamics satisfies}
\begin{align*}
\frac{\partial c^{\pm}(r,t)}{\partial t} = D_s^{\pm} \frac{\partial^2 c^{\pm}(r,t)}{\partial r^2}.
\end{align*}
Here,  $D_s^{\pm}$ is the diffusion coefficient of the active  particles of radius $R_s^\pm$. The boundary conditions of the SPM are $c^{\pm}(0,t) = 0$ and
\begin{align*}
\frac{1}{R_s^\pm}\frac{\partial c^{\pm}(r,t)}{\partial r}\Big|_{r = R_s^\pm}-\frac{c^{\pm}(R_s,t)}{{R_s^\pm}^2} = {\pm}\frac{u(t)}{D_s^{\pm}Fa^{\pm}AL^\pm},
\end{align*}
 where $F$ is Faraday's constant, $a^\pm$ is the active particle surface area, $A$ is the current collector surface area and $L^\pm$ is the electrode thickness. Applying a finite difference discretisation to the SPM's diffusion equation with the boundary conditions on a spatial domain composed of  $n+2$ equally spaced grid points separated $\Delta^{\pm} =  R_s^{\pm}/(n+1)$ apart gives
 {\small
 \begin{align*}
    \frac{\partial^2 c(r,t)}{\partial r^2}  & \approx \frac{c(r+\Delta^{\pm},t)-2c(r,t)+c(r-\Delta^{\pm},t)}{(\Delta^{\pm})^2}\,,
    \\
    \frac{\partial c(r,t)}{\partial r}\Big|_{r = R_s}  & \approx
    \frac{c(R_s^{\pm},t)-c(R_s^{\pm}-\Delta^{\pm},t)}{\Delta^{\pm}}\,.
 \end{align*}}

 Using the state-space $x_{k}(t) = c(k\Delta^{\pm},t)  $ for each electrode with $k = 1,\, \dots,\, n$, the resulting discretised SPM dynamics have the structure of~\eqref{eqn:lin_ss} with
\begin{align*}
A&=
\frac{D_s^{\pm}}{(\Delta^{\pm}) ^2}\begin{bmatrix}-2 &1 &  0 &   \\ 1& -2 &1 &  \ddots  
\\
 0 &  \ddots  & \ddots & \ddots &  0  \\    & \ddots &  1 &  -2 & 1  \\ &   & 0 & 1& -2+\frac{n+1}{n} \end{bmatrix}
, 
\\ 
B  & =  \begin{bmatrix}0, & 0, & \dots & 0, & \frac{\pm (n+1)^2}{nFa^\pm AL^\pm}\end{bmatrix}^\top,
\end{align*}
and, following~Park {\em et al.}~\cite{park2020optimal}, outputs the surface concentrations 
\begin{align*}
    y(t)= c(R_s,t) &= \frac{(n+1)}{n}x_{n}(t) \pm\frac{{R_s}^2u(t)}{nD_s^{\pm}Fa^{\pm}AL^\pm}
\end{align*}
Using the notation of a positive current in the cathode and a negative current in the anode, these discretised SPM dynamics can be seen to also satisfy the monotone and excitable assumptions. 

\begin{remark}
The above argument was based upon a finite difference discretisation of the SPM diffusion dynamics. We note that other spatial discretisation schemes, such as spectral collocation, may not preserve the monotonicity in~\ref{ls:H1_f}. Although, even then, state transformations such as~\eqref{eqn:x_trans_park} may exist which recover monotonicity. Moreover, it may be possible to further extend this monotonicity argument to more complex electrochemical battery models, such as the SPMe~\cite{moura2016battery} and the DFN models~\cite{doyle1993modeling,drummond2019feedback}, since they are also diffusion-driven. 
\end{remark}

\subsection{Thermal model}

The thermal response of a battery has been observed to play a significant role in battery fast charging, as it can impact cell degradation and safety~\cite{perez2017optimal}. Crucially, several common battery thermal models are also monotone. For example, consider the lumped thermal model~\cite{hu2019comparative} of the form
\begin{align}\label{sec:thermal}
    mC_p\dot{T}(t) = -\frac{T(t)}{R_T}+ u(t)\bigg(\sum_{j= 2}^{n}x_{j} + R_0u(t)\bigg)\,,
\end{align}
 where $m$ is the cell mass, $C_p$ is the specific heat capacity, ${T}(t) $ is the difference between the cell and room temperatures, $R_T$ is the cell's thermal resistance and where the electrical states $x_{2:n}(t)$ follow the equivalent circuit model of Section~\ref{sec:ecm}. To avoid overheating,  the cell's temperature can be constrained by
\begin{align}\label{eqn:T_lim}
    T(t) \leq \bar{T}\,,
\end{align}
where $\bar{T} >0$ is fixed. In this case, both~\ref{ls:H1_f} and~\ref{ls:H4_h} are satisfied when the current is non-negative $u(t) \geq 0$. It is clear that the state variable $T$ is ``excited'' by $u$, and hence the hypotheses~\ref{ls:H2_L} and~\ref{ls:H3_L_strict} can be checked as before, depending on~$L$.

\subsection{Li-plating constraint}

A key consideration in optimised fast-charging protocols is navigating the trade-off between cell charging time and degradation.  In fast charging,  Li-plating  has been identified as a critical degradation mechanism, because not only can it dominate cell ageing but it can also cause internal short circuits that increase the chance of a cell igniting. 

One approach to avoid Li-plating during fast charging is to add an additional constraint into the optimisation problem, as in \eqref{eq:constraint}, to ensure that the local overpotential drop across the  solid-electrolyte interface in the graphite anode always remains non-negative, as used by  Romagnoli {\em et al.}~\cite{romagnoli2019feedback} for example. In general, simulating the local overpotential drop in the anode requires a detailed {  electrochemical} model, such as the DFN model, but such models are, right now, too complex to be analysed in terms of the monotone systems of interest to this paper, as they involve a nonlinear set of partial differential algebraic equations connected across multiple domains. To avoid this, here the method of Romagnoli {\em et al.}~\cite{romagnoli2019feedback} is considered for the Li-plating constraint---a DFN model was used to generate a two dimensional look-up table to characterise when plating occurs as a function of the applied current density and the critical surface concentration of the graphite particles (one of the model states). 

\begin{figure}[h!]
         \centering
         \includegraphics[width=0.3\textwidth]{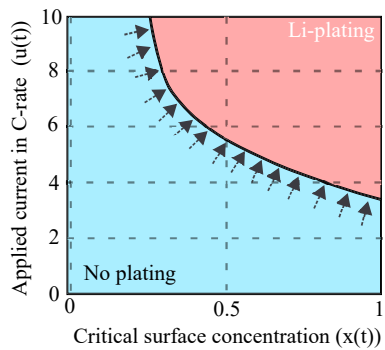}
         \caption{The lithium plating constraint of~\cite{romagnoli2019feedback} determined from the current $u(t)$ and active particle critical surface concentration $x(t)$.} 
         \label{fig:plating}
     \end{figure}

The resulting Li-plating look-up table may then be expressed graphically, as shown in Figure~\ref{fig:plating}. Here, the blue region is admissible, the red region is inadmissible, the arrows indicate the direction that is normal to the boundary of the constraint, which is always positive/non-negative.  Whilst no explicit representation of the form $h(x,u) \geq 0$ was given in~\cite{romagnoli2019feedback} for this plating constraint, Figure~\ref{fig:plating} indicates that it satisfies~\ref{ls:H4_h} since at no point along the constraint boundary (indicated by the black line) can either $x(t)$ or $u(t)$ be increased without the constraint being violated. 

\begin{figure*}
     \centering
     \begin{subfigure}[b]{0.24\textwidth}
         \centering
         \includegraphics[width=\textwidth]{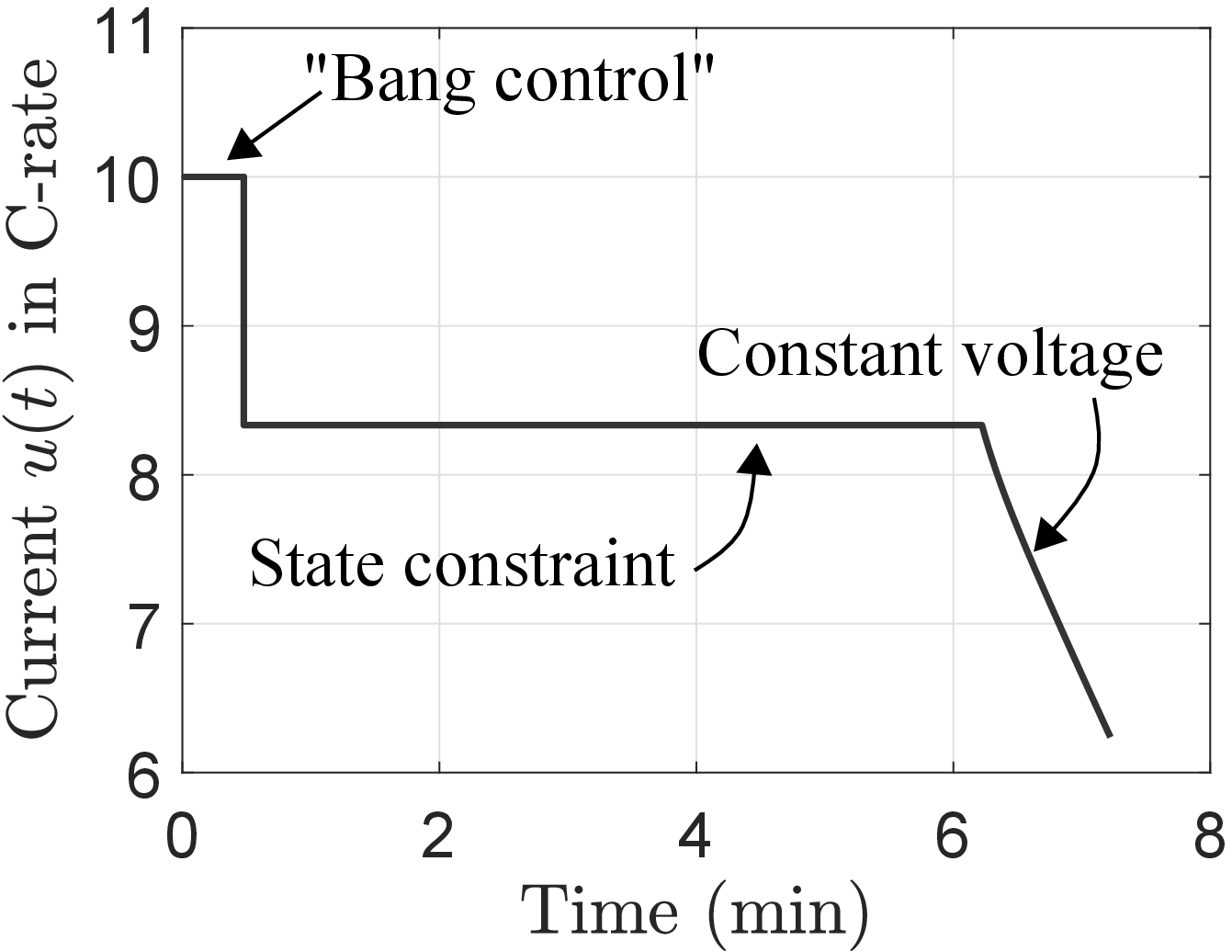}
         \caption{Charging current}
         \label{fig:y equals x}
     \end{subfigure}
     \hfill
     \begin{subfigure}[b]{0.24\textwidth}
         \centering
         \includegraphics[width=\textwidth]{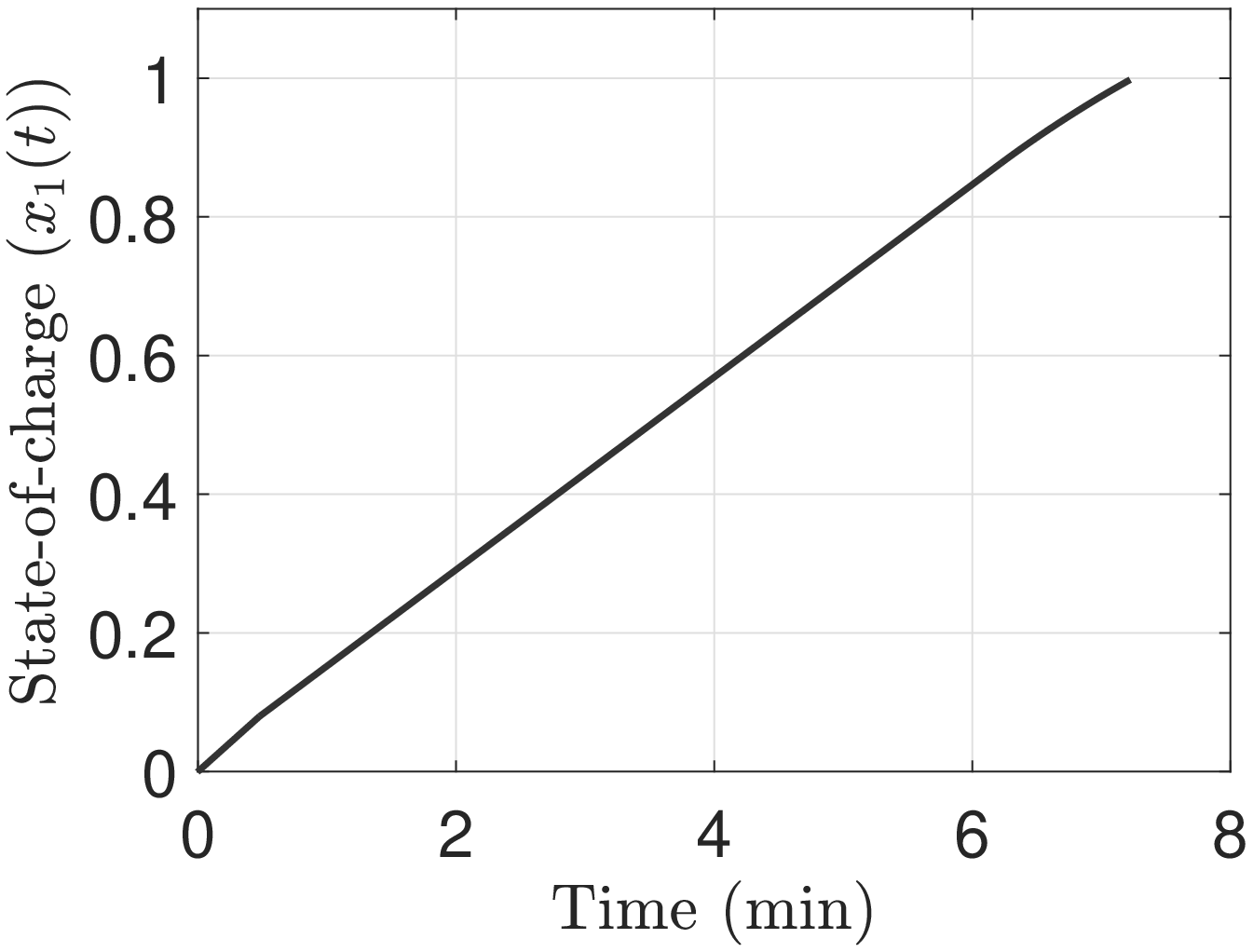}
         \caption{State-of-charge}
         \label{fig:three sin x}
     \end{subfigure}
     \hfill
     \begin{subfigure}[b]{0.24\textwidth}
         \centering
         \includegraphics[width=\textwidth]{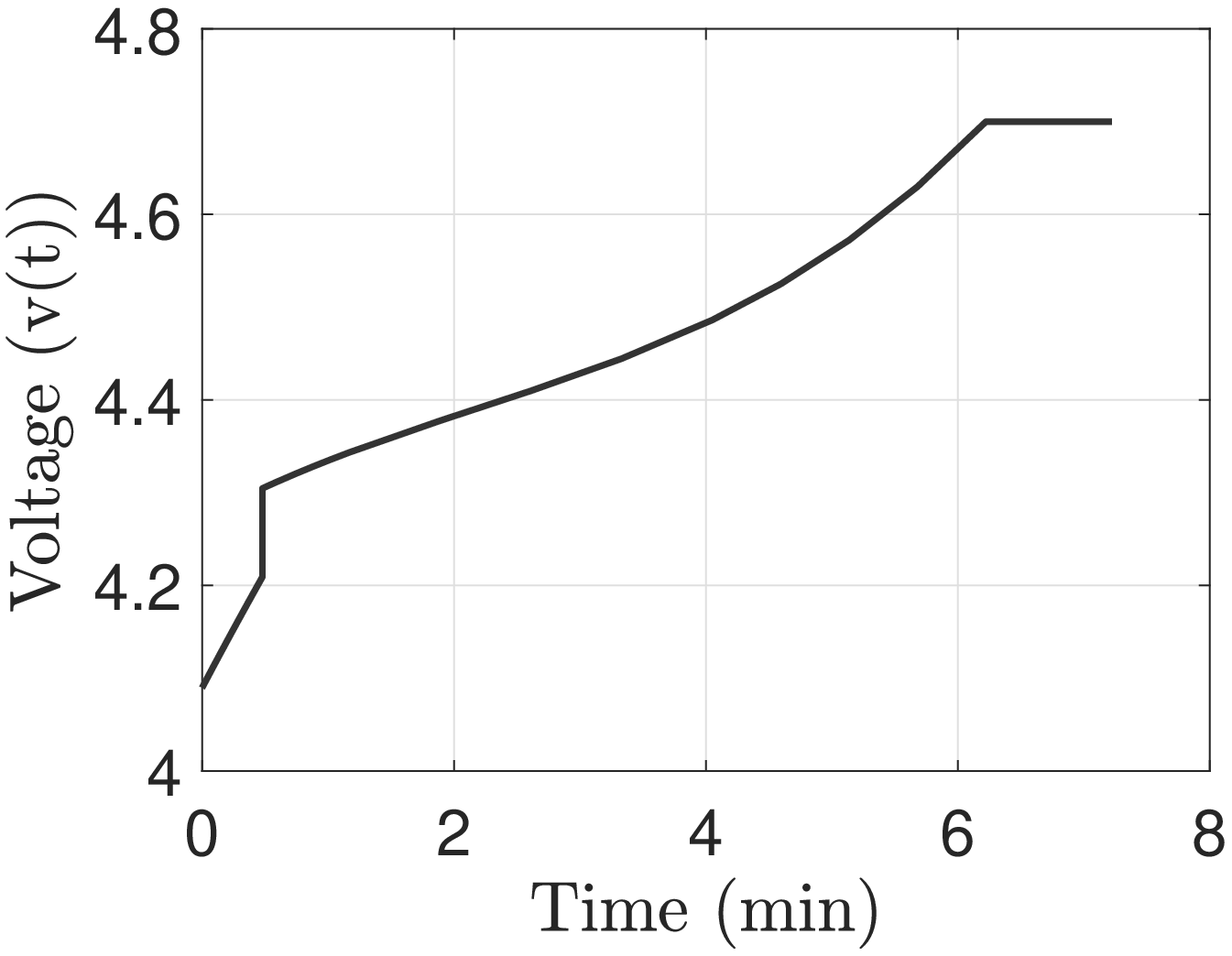}
         \caption{Voltage}
         \label{fig:five over x}
     \end{subfigure}
          \hfill
     \begin{subfigure}[b]{0.24\textwidth}
         \centering
         \includegraphics[width=\textwidth]{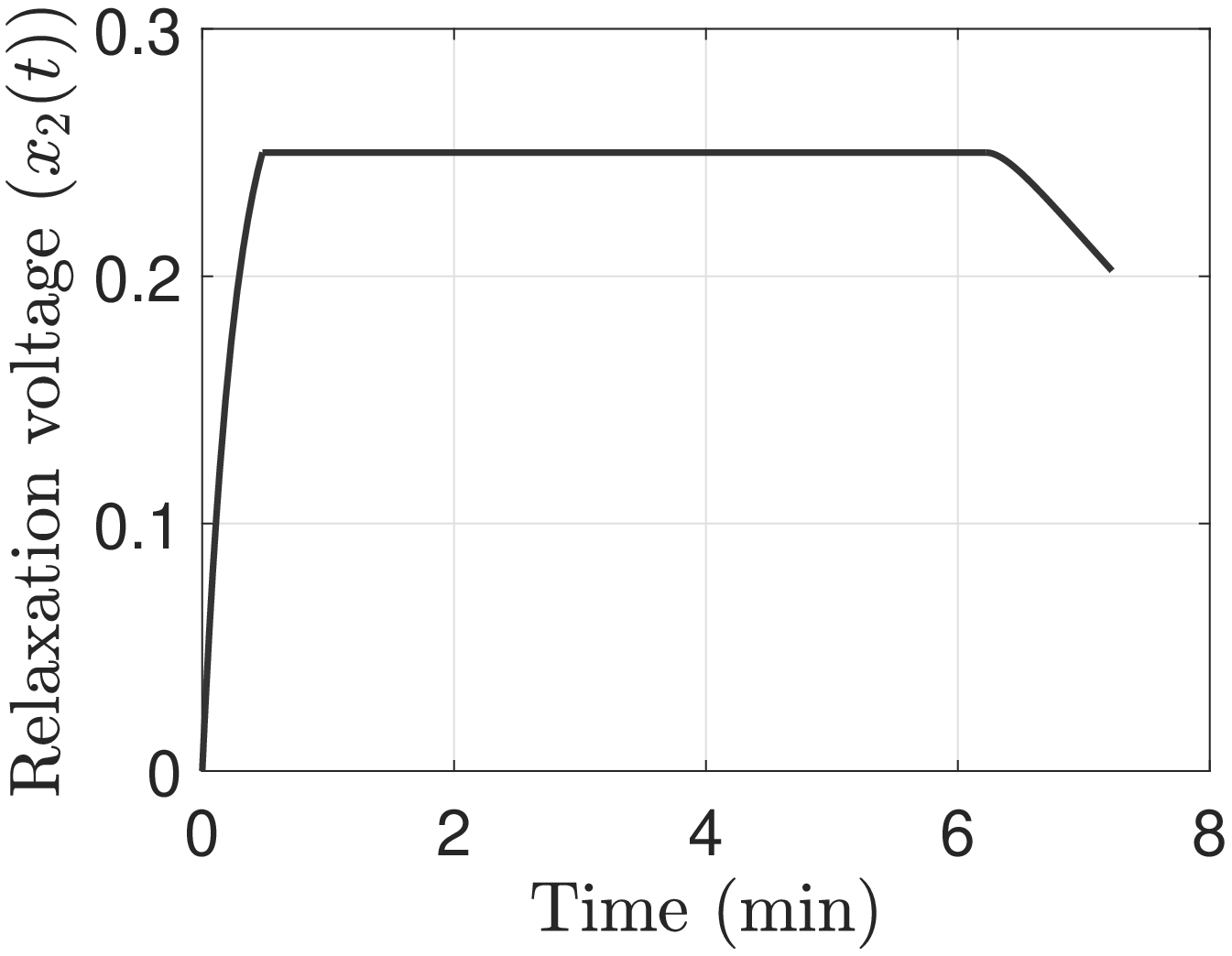}
         \caption{RC pair voltage}
     \end{subfigure}
        \caption{{  Bang-and-ride solution to the numerical example of Section {  III}. For this solution to the optimal control problem of Section I}, the obtained trajectory matches that of~\cite[Fig. 3]{ECC2022} which was computed numerically. This equivalence supports the claim of Corollary~\ref{cor:necessary} that bang-and-ride control is optimal for this class of monotone control problems.}
        \label{fig:bang_ride_sim}
\end{figure*}

\subsection{{  Cases where monotonicity is not satisfied}}

The above examples illustrate how a wide class of fast-charging problems satisfy the monotonicity assumptions~\ref{ls:H1_f}--\ref{ls:H4_h}. 
Whilst being relatively general, not all fast-charging problems are monotone. Here we explore two such situations.

(I) The monotonicity assumption of the running cost function~\ref{ls:H2_L} fails when high temperatures are penalised, such as
\[
    L(x(t),\,u(t)) = x_1(t)+u(t)-x_2(t)\,,
\]
with $x_1(t)$ being the state-of-charge and $x_2(t)$ being the cell temperature. While the temperature limits of~\eqref{eqn:T_lim} are useful for enforcing cell safety, penalising against large temperatures in the cost function  may help to reduce the onset of cell degradation caused by parasitic side-reactions, such as the growth of the solid electrolyte interface in graphite anodes.  However, the impact of high temperatures on the cell's health during fast charging is a complicated problem, as high temperatures can also reduce the likelihood of Li-plating which is one of the most significant degradation mechanisms in fast charging, as in Yang {\em et al.}~\cite{yang2019asymmetric}. The results of
~\cite{yang2019asymmetric} suggest that designing the cost function to correctly account for thermal effects in fast charging remains an open question. 

(II) The above examples all consider problems involving single cells. 
However, most energy intensive applications require battery packs composed of many individual cells connected in series and parallel, which introduces computational challenges as discussed in studies such as~\cite{pozzi2020optimal} and \cite{tanim2018fast}.
%
It is then of interest to consider whether the proposed monotonicty-based approach can be generalised to  fast charging problems of battery packs, to help  alleviate these computational issues. 

Series (also known as cascade) connections are known to preserve monotonicity, see~\cite[Proposition IV.1]{angeli2003monotone}. However, observe that bang-and-ride solutions for series connected strings will be dominated by the weakest cell, in the sense that the solution will ride the constraint associated with the lowest feasible charging current. If one of the cells is considerably weaker than the others in this series string, then the cell-to-cell variability will lead to inefficient charging.   
In contrast, connecting cells in parallel may result in a loss of monotonicity. This may be seen mathematically as a consequence of the results of~\cite{drummond2021resolving}, but is also expected in light of the observed  current ``ripples'' seen in many parallel pack simulations, such as Jocher et al.~\cite{jocher2021novel} and Drummond et al. \cite{drummond2021resolving}. These oscillations are a good indicators of a loss of monotonicity. However, the natural self-balancing of parallel connections means that as long as the cells are roughly equivalent,  it may be possible to assume that they behave as a single cell. For example, when the cells in the parallel strings are roughly equivalent, then it has been observed that the current splits across branch according to $i_k = u(t)/N$. Such an assumption is often used in practice, e.g.~in Frost et al.~\cite{frost2017completely}, and so, in this setting, monotonicity may be recovered.

\section{Numerical examples}\label{sec:num_ex}
{  In this section, a numerical example is used to illustrate the main result of Corollary~\ref{cor:necessary}--- that bang-and-ride control is optimal for control problems satisfying~\ref{ls:H1_f}--\ref{ls:H4_h}. 
Consider the equivalent circuit model of Section~\ref{sec:ecm} parameterised as in~\cite{ECC2022}. These electrical dynamics are coupled with the thermal dynamics of \eqref{sec:thermal} parameterised as in \cite{aitio2023learning}. With this model, $x_1(t)$ is the state-of-charge, $x_2(t)$ is the relaxation voltage of the circuit and $x_3(t)$ is the temperature difference between the cell and the environment. 

With these dynamics and  the initial condition $x_0 = [0,\,0,\,0]^\top$, the cost to be maximised for this fast charging problem is the integral of the state-of-charge
\begin{equation}\label{eq:cost}
 J(x_0,u) = \int_0^{\tf} x_1(t)\, \rd t\,,
\end{equation} 
Moreover, constraints on the state-of-charge, relaxation voltage, temperature current and voltage are applied
\begin{align}
       &x_1(t) \leq 1, ~
x_2(t) \leq  0.25,\\
 & x_3(t)  \leq 8, ~~
 u(t)  \leq 10 Q, v(t) \leq 4.5,
\end{align}
where $Q = 3.3\times10^3$ As is the capacitance. 

Figure~\ref{fig:bang_ride_sim} shows the results of this simulation with CC-CV charging obtained. This  solution  matches the estimated optimal solution computed via the moment-measure numerical routine of Coutier {\em et al.}~\cite{ECC2022}. The similarity of the numerically computed- and bang-and-ride-solutions: first, illustrate the result Corollary~\ref{cor:necessary} that  the constraints have to be active at some time instant during the charge. 
}

\section*{Conclusions}

A class of constrained monotone optimal control problems was studied for monotone control systems, and necessary conditions for an optimal control given. The necessary condition is essentially that, for each fixed time, a control which does not cause any constraint to engaged at all in the remaining time-window cannot be optimal. In other words, ``some constraint must be engaged at some point later''. The focus of the article was on battery fast charging and it was shown that several common battery fast-charging problems satisfy the required monotonicity assumptions. 

\bibliographystyle{IEEEtranS}
\bibliography{main_bib,Chris_bib}

\end{document}